\documentclass[12pt]{amsart}
\usepackage{amssymb}
\usepackage{amsmath}
\usepackage{bbm}
\usepackage{color}
\usepackage{verbatim}

\definecolor{darkred}{rgb}{0.75,0,0.3}

\newcommand\Af{\mathcal{A}}
\newcommand\AND{\quad\text{and}\quad}

\newcommand\ch{\mathbbm{1}}
\newcommand\ee{\mathbf e}
\newcommand\ep{\varepsilon}
\newcommand\epb{\boldsymbol{\varepsilon}}
\newcommand\Ex{\mathsf{E}}
\newcommand\Gp{\mathfrak{C}}

\newcommand\lb{\boldsymbol{\ell}}

\newcommand\Ll{\mathcal{L}}
\newcommand\Lp{\mathfrak{Lip}}

\newcommand\N{\mathbb N}
\newcommand\pib{\boldsymbol{\pi}}
\newcommand\Prob{\mathsf{Pr}}
\newcommand\R{\mathbb R}

\newcommand\Sf{\mathfrak{S}}
\newcommand\supp{\operatorname{\sf supp}}
\newcommand\taub{\boldsymbol{\tau}}
\newcommand\tb{\mathbf{t}}

\newcommand\wh{\widehat}
\newcommand\wt{\widetilde}
\newcommand\Xx{\mathcal{X}}
\newcommand\Z{\mathbb Z}
\newcommand\zero{\mathbf{0}}

\numberwithin{equation}{section}

\newtheoremstyle{mythm}
  {9pt}
  {9pt}
  {\itshape}
  {0pt}
  {\bfseries}
  {}
  { }
  {\thmnumber{(#2)}\thmname{ #1}\thmnote{ #3}}

\newtheoremstyle{mydef}
  {9pt}
  {9pt}
  {\normalfont}
  {0pt}
  {\bfseries}
  {}
  { }
  {\thmnumber{(#2)}\thmname{ #1}\thmnote{ #3}}

\theoremstyle{mythm}
\newtheorem{thm}[equation]{Theorem.}
\newtheorem{pro}[equation]{Proposition.}
\newtheorem{lem}[equation]{Lemma.}
\newtheorem{cor}[equation]{Corollary.}
\newtheorem{obs}[equation]{Observation.}

\theoremstyle{mydef}
\newtheorem{dfn}[equation]{Definition.}
\newtheorem{exa}[equation]{Example.}
\newtheorem{exs}[equation]{Examples.}

\newtheorem{rmks}[equation]{Remarks.}

\begin{document}$\,$ \vspace{-1truecm}
\title{Multidimensional random walk with reflections}
\author{\bf Judith Kloas, Wolfgang Woess}
\address{\parbox{.8\linewidth}{Institut f\"ur Diskrete Mathematik,\\ 
Technische Universit\"at Graz,\\
Steyrergasse 30, A-8010 Graz, Austria\\}}
\email{kloas@tugraz.at, woess@tugraz.at}
\date{December 6, 2016} 
\thanks{Supported by Austrian Science Fund projects FWF P24028 and W1230 
}
\begin{abstract}
Reflected random walk in higher dimension arises from an ordinary
random walk (sum of i.i.d. random variables): whenever one of the reflecting
coordinates becomes negative, its sign is changed, and the process continues from
that modified position. One-dimensional reflected random walk is quite
well understood from work in 7 decades, but the multidimensional model presents
several new difficulties. Here we investigate recurrence questions.
\end{abstract}

\maketitle

\markboth{{\sf J. Kloas and W. Woess}}
{{\sf Multidimensional random walk with reflections}}
\baselineskip 15pt

\section{Introduction}\label{sec:intro}

Let $(Y_n)_{n \ge 0}$ be a sequence of i.i.d. real valued random variables,
and let $S_n = Y_1 + \ldots + Y_n$ be the classical associated random
walk. Reflected random walk (RRW) is the process $(X_n^x)_{n \ge 0}$ given by
$$
X_0^x = x \ge 0\,, \qquad X_n^x = |X_{n-1}^x - Y_n|.
$$ 
It was first considered by {\sc von Schelling}~\cite{Sch} 
in the context of telephone networks. A rigorous examination
appeared in {\sc Feller}~\cite{Fe}, and was then developped further
by {\sc Knight~\cite{Kn}}, {\sc Boudiba}~\cite{Bo1}, \cite{Bo2} 
and {\sc Leguesdron}~\cite{Le}. The PhD Thesis of {\sc Benda~\cite{Be}}
and his unpublished papers \cite{Be1}, \cite{Be2} 
contain important contributions that will also play a role here.

Our main interest is in recurrence of this process. Positive recurrence
is settled in the above references via exhibiting a unique stationary
probability measure for the process; proving uniqueness is a non-trivial task. 
Criteria for null recurrence
were given by {\sc Smirnov~\cite{Sm}} and {\sc Rabeherimanana~\cite{Rab}}, 
and also by {\sc Peign\'e and Woess}~\cite{PeWo1}, \cite{PeWo2}.

In the present paper, we are interested in the multidimensonal variant,
where we have a random walk which is reflected 
in the first coordinate(s) and remains an ordinary random walk in the other coordinate(s).
Thus we have a probability measure
$\mu$ on $\R^{r+s}$ and the state space $\R_+^r \times \R^s$, whose elements we write
as $(x,w)$ or just $xw$, where $x \in \R_+^r$ and $w \in \R^s$. 
For $x=(x_1, \dots, x_r) \in \R^r$, we write 
$$
|x| = \bigl(|x_1|, \dots, |x_r|\bigr) \AND \|x\| = \sqrt{x_1^2 + \dots + x_r^2}\,.
$$
We consider a sequence $(Y_n\,,V_n)$ of i.i.d. $\mu$-distributed random vectors
with $Y_n \in \R^r$ and $V_n \in \R^s$. Then our process starting at
$(x,w)$ is given by
\begin{equation}\label{eq:process}
(X_n^x\,, w + Z_n)\,, \quad\text{where}\quad X_0^x = x,\; X_n^x = | X_{n-1}^x - Y_n|\,,
\AND Z_n = V_1 + \dots + V_n\,.
\end{equation}
We shall usually start with $w=0$.
For studying transience / recurrence, only the cases $s \in \{ 0, 1, 2\}$ are of interest,
since otherwise already $(Z_n)$ is transient.

We remark immediately that the process \eqref{eq:process} factorises in each coordinate.
\begin{itemize}
 \item 
If $i \le r$, then the $i$-th coordinate of $(X_n^x\,, v + Z_n)$ is the
reflected random walk on $\R_+$ which starts at $x_i$ and is driven by the $i$-th
marginal $\mu_i$ of $\mu$.
 \item
If $r+1 \le i \le r+s$ then the $i$-th coordinate is the random walk (sum of
i.i.d. random variables) which starts at $v_i$ and whose law is the $i$-th
marginal of $\mu$.
 \item
In particular, $(X_n^x)$ is the reflected random walk on $\R_+^r$ driven by $\mu_{\lfloor r \rfloor}$,
the overall marginal of $\mu$ on the first $r$ coordinates, and $(v+Z_n)$ is the
(ordinary) random walk on $\R^s$ whose law $\mu_{\lceil s \rceil}$ is the overall marginal
of $\mu$ on the last $s$ coordinates.
\end{itemize}

As usual, we shall distinguish between the \emph{lattice} and the \emph{non-lattice} 
cases in each coordinate.
The latttice case arises when there is $\kappa > 0$ such that 
$\supp(\mu_i) \subset \kappa \cdot \Z$.
In this case, we can and will always assume without loss of generality that 
\begin{equation}\label{eq:gcd}
\supp(\mu_i) \subset \Z \AND \gcd \supp(\mu) =1.
\end{equation} 
The marginal $\mu_i$ is non-lattice if no $\kappa$ as above exists.

Thus, we shall assume that $r = r_1 + r_2$ and $s=s_1 + s_2$ such that
the marginals $\mu_i$ satisfy \eqref{eq:gcd} for $i=1, \dots, r_1$ and
$i = r+1, \dots, r+s_1\,$, while they are non-lattice in the other
coordinates. Consequently, it is natural that we restrict our state space to 
\begin{equation}\label{eq:statespace}
\Xx = \N_0^{r_1}\times \R_+^{r_2} \times \Z^{s_1} \times \R^{s_2}\,.
\end{equation}

In the non-discrete situation, our study of recurrence and stationary probability 
distributions focusses on topological recurrence.

One of our basic tools is \emph{local contractivity,} a property of stochastic
dynamical systems that was introduced by {\sc Babillot, Bougerol and Elie}~\cite{BaBoEl}
and studied in detail by {\sc Benda}~\cite{Be}. We summarise the basic facts in the
short \S 2. In \S 3, we review one-dimensional reflected random walk and display the smart
method of \cite{Be} in the lattice case to induce a locally contractive process 
on the even numbers (Proposition \ref{pro:const}). We also display an example 
of a transient reflected random walk where the
non-reflected walk is recurrent.  

In \S 4, we consider the multimdimensional case with
reflection in all coordinates. The main result is Theorem \ref{thm:posrec},
characterising positive recurrence.  While the case where all marginals are non-lattice
is covered by {\sc Peign\'e}~\cite{Pe}, the presence of lattice marginals leads to
considerable additional difficulties which we elaborate in detail. Subsequently, 
we provide several partial results and examples regarding the null-recurrent situation, 
where however a complete characterisation remains a challenging open problem.

In the last \S 5, we consider the general situation where some coordinates are reflected
and others (at most 2) are ``free'' (non-reflected). Our second main result is
Theorem \ref{thm:r-s-rec}, where we assume that the reflected part is (topologically)
recurrent and the non-reflected coordinates are centred and satisfy the
natural moment conditions. While this is easy when the reflected part is discrete (lattice),
additional tools from Ergodic Theory are needed in general, invoking results on
recurrence of stationary random walks which are due to {\sc Atkinson}~\cite{At} and
{\sc Schmidt}~\cite{Schm}.  This leads to recurrence of the process. 
Again, it is a challenging open problem to
handle the case when the reflected part is only null-recurrent.

\section{A summary on local contractivity}\label{sec:local}

We recall a few facts that were explained in \cite{PeWo2}, plus additional
features. Unless otherwise stated, the facts displayed in this
section can be found in \cite{PeWo2}, resp. the remarkable PhD thesis
\cite{Be}.

In general, we consider a proper metric space $(\Xx,d)$  
and the monoid $\Gp(\Xx)$ of all continuous mappings $\Xx \to \Xx$.
It carries the topology of uniform convergence on compact sets.
Now let $\wt\mu$ be a Borel probability measure on $\Gp(\Xx)$, and let 
$(F_n)_{n \ge 1}$ be a sequence of i.i.d. $\Gp(\Xx)$-valued random variables
(functions) with common distribution $\wt\mu$, defined on a suitable 
probability space $(\Omega, \Af, \Prob)$. The associated 
\emph{stochastic dynamical system (SDS)} $\omega \mapsto X_n^x(\omega)$ is given by
\begin{equation}\label{eq:SDS}
X_0^x = x \in \Xx\,,\AND X_n^x = F_n \circ F_{n-1} \circ \dots \circ F_1(x)\,,\quad n \ge 1\,.
\end{equation}

In case of reflected random walk on $\Xx = \N_0^{r_1}\times \R_+^{r_2}$ 
(as  in \eqref{eq:statespace} with $s_1=s_2=0$), we have
$F_n(x) = |x-Y_n|$, and these mappings are contractions,
whence we may replace $\Gp(\Xx)$ by the closed sub-monoid $\Lp_1(\Xx)$
of all Lipschitz mappings with Lipschitz constant~$\le 1$. If $\mu$ is the distribution 
on $\R^d$ of the increments $Y_n$, then $\wt \mu$ is the image of $\mu$ under the
mapping $\R \to \Lp_1(\Xx)$, $y \mapsto f_y$, where $f_y(x) = |x-y|$.

\begin{dfn}\label{def:loccont} The SDS is called \emph{locally contractive,} 
if for every $x\in \Xx$ and every compact $K \subset \Xx$, 
$$
\Prob[d(X_n^x,X_n^y)\, \cdot \ch_K(X_n^x)\to 0 
\quad\text{for all}\; y \in \Xx] = 1\,.
$$
It is called \emph{strongly contractive,} if for 
every $x \in \Xx$, 
$$
\Prob[d(X_n^x,X_n^y)\to 0 \quad\text{for all}\; y \in \Xx] = 1\,.
$$
\end{dfn}

\begin{pro}\label{pro:rec} A locally contractive SDS is either 
\emph{transient,} 
$$
\Prob[d(X_n^x,x) \to \infty] = 1 \quad\text{for every}\;x \in \Xx\,
$$
or (topologically) \emph{recurrent} in the sense that there is a
maximal non-empty closed subset $\Ll \subset \Xx$ with the property that
for every open set $U$ that intersects $\Ll$,
$$
\Prob[X_n^x \in U \;\text{infinitely often}] = 1 \quad\text{for every}\;x \in \Xx.
$$
In the recurrent case, $\Ll$ coincides almost
surely with the set of accumulation points of any trajectory $\bigl(X_n^x(\omega)\bigr)$.

$\Ll$ is also characterised as the 
smallest non-empty closed subset of $\Xx$ with the
property that $f(\Ll) \subset \Ll$ for every $f \in \supp (\wt \mu) \subset \Gp(\Xx)$.
\end{pro}

Note that the last characterisation does not rely on recurrence; it depends
only on $\supp (\wt \mu)$.
In the recurrent case, the set $\Ll$ is called the \emph{attractor,} 
and the SDS is strongly contractive. 

\smallskip
An \emph{invariant measure} for an SDS is a Radon measure $\nu$ on $\Xx$
such that for any Borel set $B \subset \Xx$,
$$
\int \ch_B(X_1^x) \, d\nu(x) = \nu(B).
$$ 
Part (a) of the following is obvious; for (b) see \cite{PeWo1}.

\begin{pro}\label{pro:invmeas} \emph{(a)} A locally contractive 
SDS which has an invariant probability measure is recurrent.
\\[3pt]
\emph{(b)} A locally contractive SDS which is recurrent
has an invariant measure $\nu$ which is unique up to multiplication by constants.
In this case, the following holds.
\begin{itemize}
\item $\quad \supp(\nu) = \Ll$.
\item $\quad \nu(\Ll) < \infty$ if and only if the SDS 
is positive recurrent\\ \hspace*{7pt} (the return time to any open set
which intersects $\Ll$ has finite expectation).
\end{itemize}
\end{pro}

For an SDS of contractions,
let $\Sf(\wt \mu)$ be the sub-semigroup of $\Lp_1(\Xx)$ generated by 
$\supp(\wt \mu)$ and $\overline\Sf(\wt \mu)$ its closure. 

\begin{pro}\label{pro:const}  A non-transient SDS of contractions is locally contractive
if and only if $\overline\Sf(\wt\mu)$ contains a constant function.
In this case, it is recurrent as well as strongly contractive, so that it is
absorbed by the attractor: for any starting point $x$, 
$$
d(X_n^x,\Ll) \to 0 \quad \text{almost surely.}
$$
\end{pro}

See \cite{Le}, \cite{Pe}, \cite{Be} and \cite[Theorem 4.2]{PeWo1}.
An important tool is going to be the following.

\begin{pro}\label{pro:reverse} Suppose that our SDS of contractions is locally contractive and
 has an invariant probability measure $\nu$. Then there is an $\Xx$-valued
random variable $Z$ such that for any starting point $x \in X\,$,
$$
\wh X_n^x = F_1 \circ \dots \circ F_n(x) \to Z \quad \text{almost surely.}
$$
The distribution of $Z$ is $\nu$.
\end{pro}

Note that in general, $(\wh X_n^x)_{n\ge 0}$ is not Markovian.
The proposition is proved in \cite{Le} and \cite{Pe} under the assumption that $\Xx =\R^r$. 
In \cite{Pe}, it concerns more general SDS of contractions which are not necessarily compositions 
of i.i.d. mappings, but driven by a positive recurrent Markov chain. 
It readily extends to any proper metric space $\Xx$ in place of $\R^r$. 

\section{A review of one-dimensional reflected random walk}\label{sec:dimone}

Here, the $Y_n$ are real random variables with common distribution $\mu$. We always assume
that
\begin{equation}\label{eq:nontriv}
\mu\bigl((0\,,\,\infty)\bigr) > 0\,.
\end{equation}
The state space is $\Xx = \R_+$ in the non-lattice case, and $\Xx = \N_0 = \{ 0, 1, \dots\}$
in the lattice case \eqref{eq:gcd}.

\medskip

\noindent\textbf{A. Irreducibility and local contractivity}
$\,$

\smallskip

\noindent
We set
\begin{equation}\label{eq:Ll}
\begin{aligned}
N &= \sup \supp(\mu)\,,\quad \text{if}\quad \supp(\mu) \subset \R_+\,,\quad\text{resp.}\quad
N = \infty\,,\quad \text{otherwise, and}\\
\Ll &= [0\,,\,N] \cap \Xx\,, \quad\text{if}\quad N < \infty\,,\quad\text{resp.}\quad
\Ll = \R_+ \cap \Xx\,, \quad\text{if}\quad N = \infty\,.
\end{aligned}
\end{equation}
Then $(X_n^x)$ is (topologically) irreducible  
on $\Ll$, see \cite{Le}, \cite{Bo2}, \cite{Rab}, \cite{PeWo1}, \cite{PeWo2}.
Regarding local contractivity, the following is known;
compare with \cite{Le}, \cite{Be}, \cite{Be1}, \cite{PeWo1} and \cite{PeWo2}.
 
\begin{pro}\label{pro:contractive} Assume that $\mu$ is non-lattice and satisfies 
\eqref{eq:nontriv}. 
Then the reflected random walk induced by $\mu$ is locally contractive.  
\end{pro}

In the lattice case, we cannot have local contractivity. Indeed, if $x, y \in \N_0$ then
$X_n^x-X_n^y$ always has the same parity as $x-y$. However, the PhD thesis \cite{Be} contains
a smart observation \& method which we now explain. For the remainder of this sub-section,
we assume that $\mu$ satisfies \eqref{eq:gcd}. 

For $x \in \Z$, let $\pi(x) = 0$ if $x$ is even, and $\pi(x) = 1$ if $x$ is odd. Then the 
following is obvious.

\begin{lem}\label{lem:evodd} The process $\bigl(\pi(X_n^x)\bigr)_{n \ge 0}$ is a Markov chain
on $\{0,1\}$ with transition probabilities
$$
p(0,0) = p(1,1) = \mu(2 \cdot\Z) \AND p(0,1) = p(1,0) = \mu(2 \cdot\Z + 1)\,.
$$
\end{lem}
In particular, it depends only on the parity of the starting point $x$, and by \eqref{eq:gcd}
it is irreducible. It is therefore positive recurrent, the return times to each of the
two states coincide, their distribution is easily computed, and the expected value is $2$.
We can consider the induced process on $2\cdot \N_0\,$, resp. on $2\cdot \N_0 + 1$. 
That is, we consider the a.s. finite stopping times 
\begin{equation}\label{eq:stop}
 \begin{aligned}
\tb(0)&=0\,, \quad \text{and, setting} \quad S_k = Y_1 + \dots + Y_k\,,\\ 
\tb(n)& = \inf \{ k > \tb(n-1) : \pi(X_k^x) = \pi(x) \} =  \inf \{ k > \tb(n-1) : S_k \; \text{is even}\, \}\,.
 \end{aligned}
\end{equation}
No matter whether the starting point of $(X_n^x)$ is even or odd,
the induced process $(X_{\tb(n)}^x)$ on the respective class $2\cdot \N_0$ or $2 \cdot \N_0 + 1$ is again
an SDS generated by i.i.d. contractions:
\begin{equation}\label{eq:induced}
X_{\tb(n)}^x = F_n \circ F_{n-1} \circ \dots \circ F_1(x) \quad \text{with} \quad
F_n = f_{Y_{\tb(n)}} \circ f_{Y_{\tb(n)-1}} \circ \dots \circ f_{Y_{\tb(n-1)+1}}\,.
\end{equation}

Let $\wt \mu_{\tb}$ be the distribution of $F_1$ on $\Lp_1(\N_0)\,$. Since the proof of 
the following is not easily accessible \cite{Be}, we present it here.

\begin{pro}\label{pro:even} If $\mu$ satisfies \eqref{eq:gcd} then 
$\ch_{2\cdot \N_0} \in \overline\Sf(\wt \mu_{\tb})$. 
\end{pro}

\begin{proof} 
Recall the notation $f_y(x) = |x-y|$.
\\[3pt]
\emph{Step 1.} There are elements $y_0\,,\dots, y_m \in \N$ such
that 
$$
0 < y_0 < \dots < y_m\,,\quad \gcd\{y_0\,,\dots, y_m\} = 1\,,\AND f_{y_k} \in \Sf(\wt \mu).
$$
Indeed, there is $b \in \supp(\mu)$ with $b \ge 1$, and if $a < 0$ then 
$a' = a + (\lfloor -a/b \rfloor +1)b \ge 1$, and we check easily that
$$
f_{a'} = f_b^{\lfloor -a/b \rfloor +1} \circ f_a\,,
$$
whence $f_{a'} \in  \Sf(\wt \mu)$ whenever $a \in \supp(\mu)$. Now,
there are $a_1\,,\dots\,,a_n \in \supp(\mu) \setminus \{0\}$ with greatest
common divisor $1$. We replace each $a_k < 0$ by $a_k'$ and add $b$ to
the updated collection of elements. Then we order them and elminate possibly 
redundant ones to get $y_0\,,\dots, y_m\,$.
\\[3pt]
\emph{Step 2.} We now set $d_k = \gcd\{y_0\,,\dots,y_k\}$, so that 
$y_0 = d_0 > d_1 > \dots > d_m = 1$. We construct recursively 
elements $g_0\,,\dots, g_m \in \Sf(\wt \mu)$ such that
$$
g_k(n) = f_{d_k}(n) \quad\text{for all } \; n \in \{0,1,\dots, d_k\}\,.
$$
We start with $g_0 = f_{y_0}$. If we already have $g_{k-1}$ then we follow the
steps of the Euclidean algorithm posing $r_0= y_k\,$, $r_1=d_{k-1} < y_k$ 
and applying repeated integer division $r_{i-1} = q_ir_i + r_{i+1}$ 
with $0 \le r_{i+1} < r_i$. If $j$ is the first index for which $r_{j+1}=0$
then $r_j = \gcd\{y_k\,,d_{k-1}\}= d_k$. We let
$$
h_0 = f_{y_k}\,,\; h_1 = g_{k-1}\,,\AND h_i = h_{i-1}^{q_{i-1}} \circ h_{i-2}\,,\; i =2,\dots, j\,.
$$
Then we set $g_k = h_j\,$. (The $h_i$ as well as $j$ depend on $k$.) One checks easily that
also $g_k$ has the proposed properties.
\\[3pt]
\emph{Step 3.} We now have $g_m(n) = f_1(n)$ for $n \in \{0,1\}\,$.
Thus $g_m$ sends even numbers to odd ones and vice versa, and since it is
a contraction, this implies that $|g_m(n+1) - g_m(n)| = 1$ for all $n$.
From this we deduce inductively that for all $n \in \N$,
$$
g_m(2n-1) \in \{ 0, 2, \dots, 2n-2\} \AND g_m(2n) \in \{ 1, 3, \dots, 2n-1\}.
$$
Therefore $h = g_m^2 \in \Sf(\wt \mu)$  preserves the parity of any $n \in \Z$. 
But this just means that $h \in  \Sf(\wt \mu_{\tb})$. The above yields that
$$
h^k(2n-1) = 1 \AND h^k(2n-2) = 0 \quad\text{for}\quad n=1, \dots, k.
$$
As $k \to \infty$, we see that $h^k \to \ch_{2\cdot\N_0+1}$ pointwise, so that
$\ch_{2\cdot\N_0+1} \in \overline \Sf(\wt \mu_{\tb})$. 
\end{proof}

\begin{cor}\label{cor:induced}
 The induced process $(X_{\tb(n)}^x)$ is locally contractive on each of
the classes $2\cdot \N_0$ and $2\cdot \N_0 +1$.
The respective limit sets are $\Ll_0=\Ll \cap (2\cdot \N_0)$, resp. $\Ll_1=\Ll \cap (2\cdot \N_0+1)$,
where $\Ll$ is as in \eqref{eq:Ll}.
\\[3pt]
If the originial reflected random walk $(X_n^x)$ is positive, resp. null recurrent, then
so is the induced process on each of the two classes, and $X_n^x-X_n^y \to 0$ a.s.
whenever $x-y$ is even. 
\end{cor}

The statement on recurrence is clear from the fact that the return time to the starting
point of the induced process is bounded by the return time of the original process.
We remark that \cite{Be} has general results in the same spirit, where the 
SDS has a finite, irreducible factor chain.

\medskip

\noindent\textbf{B. Non-negative $Y_n\,$}
$\,$

\smallskip

\noindent
We first consider the situation when $Y_n \ge 0$ (of course excluding 
the trivial case $Y_n \equiv 0$), so that the increments of $(X_n^x)$ are
non-positive except possibly at the moments of reflection. In this case,
{\sc Feller~\cite{Fe}} and {\sc Knight~\cite{Kn}} have computed an invariant measure 
for the process when the $Y_n$ are non-lattice random variables, while 
{\sc Boudiba~\cite{Bo1}}, \cite{Bo2} has provided such a measure
when the $Y_n$ are lattice variables.

\begin{lem}\label{lem:invmeas} Suppose that $\supp{\mu} \subset [0\,,\infty)$.
\\[3pt]\noindent
\emph{(a)} If $\mu$ is non-lattice then an invariant measure is given by
$$
\nu(dx) = \mu\bigl((x\,,\,\infty)\bigr)\,dx\,.
$$
\emph{(b)} If $\mu$ is lattice,  then an invariant measure is
$$
\nu(0) = \frac{1-\mu(0)}{2} \AND
\nu(x) = \frac{\mu(x)}{2} + \mu\bigl((x\,,\,\infty)\bigr)\,,\quad\text{if}
\;\;x\in \N\,.
$$ 
\end{lem}
In both cases, $\nu\bigl([0\,,\infty)\bigr) = \Ex(Y_1)$. This leads to the following well-known
property.

\begin{cor}\label{cor:cond0} Reflected random walk is positive recurrent on $\Ll$ if and only if 
\hbox{$\Ex(Y_1) < \infty$.}
\end{cor}

The next question is when we have null-recurrence. The following sufficient conditions
are due to \cite{Sm}, \cite{PeWo1} and \cite{Rab} (in this order). 

\begin{pro}\label{pro:refl-recurr} Suppose that $\supp(\mu) \subset \R_+$.
Then each of the following conditions implies the next one and is sufficient
for recurrence of the reflected random walk on $\Ll$.
\begin{gather}
\Ex\bigl(\sqrt{Y_1}\,\bigr) < \infty \tag{i}\label{cond1}\\
\int_{\R^+} \mu\bigl((x\,,\,\infty)\bigr)^2\,dx < \infty
\tag{ii}\label{cond2}\\
\lim_{y \to \infty} \mu\bigl((y\,,\,\infty)\bigr)
\int_0^y \mu\bigl((x\,,\,y]\bigr)\,dx = 0 \tag{iii}\label{cond3}
\end{gather}
(In the lattice case, the integrals reduce to sums and $dx$ is the counting measure on $\N_0\,$.)
\end{pro}
\medskip

\medskip

\noindent\textbf{C. Two-sided increments}
$\,$

\smallskip

We now drop the assumption that $Y_n \ge 0$. Of course, we require that $\mu$
is such that we do not have $S_n = Y_1 + \dots + Y_n \to -\infty$ with positive 
probability (= probability
1 by Kolmogorov's 0-1 law), because
in this case there are only finitely many reflections, and 
$X_n^x \to \infty$ almost surely.

Let $Y_n^+ = \max \{Y_n, 0 \}$ and $Y_n^- = \max \{-Y_n, 0\}$.
If \emph{(a)} $\;\Ex(Y_1^-) < \Ex(Y_1^+) \le \infty\,$, or if
\emph{(b)} $\;0 < \Ex(Y_1^-) = \Ex(Y_1^+) < \infty\,$, then
$\limsup S_n = \infty\,$ almost surely, so that there are infinitely many
reflections. 

We now assume that $\limsup S_n = \infty$ almost surely.
Then the (non-strictly) ascending  \emph{ladder epochs}
$$
\lb(0)  = 0\,,\quad  \lb(k+1) = \inf \{ n > \lb(k) : S_n \ge S_{\lb(k)} \}
$$ 
are all almost surely finite, and the random variables $\lb(k+1) - \lb(k)$ 
are i.i.d.
We can consider the \emph{embedded random walk} $S_{\lb(k)}\,$, $k \ge 0$, 
which tends to $\infty$ almost surely. Its increments 
$\overline Y_k = S_{\lb(k)} - S_{\lb(k-1)}\,$, $k \ge 1$, are i.i.d. non-negative
random variables with distribution denoted $\overline{\mu}$. 
Furthermore, if $(\overline{\!X}_k^x)$ denotes the reflected random
walk associated with the sequence $(\overline Y_k)$, while $X_n^x$ is our original
reflected random walk associated with $(Y_n)$, then 
$$
\overline{\!X}_k^x = X_{\lb(k)}^x\,,
$$ 
since no reflection can occur between times $\lb(k)$ and $\lb(k+1)$.
It is easy to see that the embedded reflected random walk $(\overline{\!X}_k^x)$ is recurrent if and 
only the original reflected random walk is recurrent. This leads to the following
sufficient recurrence criteria \cite{PeWo2}.

\begin{pro}\label{pro:sqrt}
Reflected random walk  $(X_n^x)$ is (topologically) recurrent on $\Ll$, if
\begin{quote} 
\emph{(a)} $\;\Ex(Y_1^-) < \Ex(Y_1^+) \le \infty$ and 
$\Ex\bigl(\sqrt{Y_1^+}\,\bigr) < \infty\,$, 
or if\\
\emph{(b)} $\;0 < \Ex(Y_1^-) = \Ex(Y_1^+)$ and 
$\Ex\Bigl(\sqrt{Y_1^+}^{\,3}\Bigr) < \infty\,$. 
\end{quote}
In case (a), one has positive recurrence if and only if $\Ex(Y_1^+) < \infty$, 
and in case (b), one has null recurrence.
\end{pro}

In the positive recurrent case of (a), we also explain how to get the
invariant probability measure from the one for the embedded process. Write
$\nu$ for the latter. It is computed from $\overline\mu$ according
to Lemma \ref{lem:invmeas}.  
For any Borel set $B \subset \R$,
 \begin{equation}\label{eq:lift-meas}
 \nu(B) = \int_{\Ll} \Ex \left( \sum_{k=0}^{\lb(1)-1} \ch_B(X_k^x) \right)\, d\nu(x)\,,
\end{equation}
and it is finite because $\lb(1)$ has finite expectation. (Note that for $k < \lb(1)$ we have
$X_k^x = x - S_k\,$.) 
Among the observations from \cite{Be} and \cite{PeWo2}, we 
also recall the following.

\begin{lem}\label{lem:symmetric}
If $\mu$ is symmetric on $\R$ (resp. $\Z$), then reflected random walk 
is (topologically) recurrent  if and only if the random walk $(S_n)$ is recurrent.

In particular, if $\mu$ is symmetric and has finite first moment, then
the associated reflected random walk is recurrent.
\end{lem}

The last statement follows from the classical result that when $\Ex(|Y_1|) <\infty$ and
$\Ex(Y_1)=0$ then $S_n$ is recurrent; see {\sc Chung and Fuchs~\cite{ChFu}}.

At this point we can ask whether also in the non-symmetric case, recurrence of 
the ordinary random walk $(S_n)$ always implies recurrence of the associated reflected
random walk. The answer is ``no'', as the following example shows.

\begin{exa} Let the $Y_n$ be i.i.d. with centred distribution $\mu$ supported by 
$\{ k \in \Z : k \ge -1\}$,
and $\overline\mu$ the distribution of the $\overline Y_k$. 
By Wiener-Hopf-factorisation as in \cite{Fe} (see \cite{PeWo2} in the present
context),
$$
\mu = \overline \mu + \delta_{-1} - \overline\mu * \delta_{-1}\,,
$$
because $\delta_{-1}$ is the first strictly descending ladder distribution
associated with $\mu$. Thus, we have
$$
\mu(-1) = 1-\overline \mu(0) \AND \mu(x) = \overline \mu(x) - \overline\mu(x+1) \quad \text{for } x \in \N_0\,.
$$ 
If we \emph{start} with a probability measure $\overline \mu$ on $\N_0$ which satisfies 
$\overline \mu(x) \ge \overline\mu(x+1)$ for all $x$ then we can \emph{construct} $\mu$ 
in this way, whence $\mu$ has finite first moment and is
centred. By the uniqueness of the Wiener-Hopf decomposition, $\overline\mu$ is indeed the
first ascending ladder distribution of $\mu$. Now define 
$\overline \mu(x) = c\, \log(x+2)/(x+2)^{3/2}$, $x \in \N_0\,$. Then the random walk $(S_n)$
with law $\mu$ is recurrent.
But by \cite[Ex. 5.11]{PeWo2}, resp. its discrete variant in \cite{PeWo1},
the embedded reflected random walk is transient, and so is the reflected random
walk induced by $\mu$.
\end{exa}

\section{Reflection in all coordinates}\label{sec:only-refl}

In this section, we study the multidimensional case \eqref{eq:process} with 
$r = r_1 + r_2 \ge 2$ and $s=0$. Our state space is $\Xx = \N_0^{r_1}\times \R_+^{r_2}$.
We suppose that all one-dimensional marginals of the probability measure $\mu$
satisfy \eqref{eq:nontriv}.
Suppose initially that $r_1 \ge 1$. For $x = (x_1\,,\dots,x_r) \in \Xx$, write
$$
X_n^x = (X_{n,1}^{x_1}\,,\dots, X_{n,r}^{x_r})\,,
$$
so that $(X_{n,i}^{x_i})_{n \ge 0}$ is the reflected random walk 
induced by $\mu_i\,$. When the latter is recurrent on its unique essential class, 
we know from propositions \ref{pro:contractive} and \ref{pro:const} 
that $X_{n,i}^{x_i} - X_{n,i}^{y_i} \to 0$ almost surely, when $i > r_1$ and 
$x_i, y_i \in \R_+$ are arbitrary. On the other hand, when $i \le r_1$, by Corollary 
\ref{cor:induced} the same holds as long as $x_i, y_i \in \N_0$ have the same parity.
Recall the mapping $\pi(k) = \ch_{2\cdot \Z+1}(k)$ and define
$$
\pib: \Xx \to \{0,1\}^{r_1}\,,\quad \pib(x_1, \dots, x_r) = \bigl(\pi(x_1), \dots, \pi(x_{r_1})\bigr).
$$
Then recurrence of the marginal processes implies
\begin{equation}\label{eq:multi-contr}
\|X_n^x - X_n^y\| \to 0 \quad \text{almost surely, whenever}\quad \pib(x)=\pib(y)\,.
\end{equation}
For an element $\epb$ of the hypercube $\Z_2^{r_1}=\{0,1\}^{r_1}$, let 
$$
\Xx_{\epb} = \{ x \in \Xx : \pib(x)=\epb \}\,.
$$
We note that $\pib(X_n^x) = \pib(x + S_n)$, where again $S_n=Y_1 \dots + Y_n\,$.
The process $\bigl(\pib(X_n^x)\bigr)_{n \ge 0}$ is a random walk on the hypercube
which is translation invariant with respect to addition mod $2$ and automatically symmetric.
It is driven by the probability measure $\pib\mu(\epb) = \mu\bigl((2\cdot \Z)^{r_1} + \epb\bigr)\,$.
Since by assumption \eqref{eq:gcd}, each $\supp(\mu_i)$, $i \le r_1\,$, contains odd elements,
$\pib\mu$ charges elements different from $\zero=(0, \dots, 0)$.  The random walk is not 
necessarily irreducible; the group $\Z_2^{r_1}$ decomposes into
a subgroup $\Gamma$ (consisting of $\zero$ and the elements that can be reached from $\zero$) 
and its cosets, on each of which that random walk is irreducible. This leads us to
the following.

\begin{obs}\label{obs:cosets}
Let $\Gamma^{(j)}$, $j=1, \dots, 2^d$, be the cosets of $\Gamma$ in $\Z_2^{r_1}$.
Then $1 \le d < r_1$, and our state space decomposes into the classes
$$
\Xx^{(j)} = \bigcup_{\epb \in \Gamma^{(j)}} \Xx_{\epb}\,,
$$
so that reflected random walk started in some $x \in \Xx^{(j)}$ never exits from
that class.
\end{obs}

Thus, even though all marginal one-dimensional reflected walks are (topologically)
irreducible on the respective sets $\Ll_i$ ($i=1, \dots, r$), the multidimensional 
reflected random walk may 
have a decomposition into non-interacting parts. We shall see an example further below;
in particular, the structure of the essential class(es) is not as simple
as in the one-dimensional case \eqref{eq:Ll}.
Of course, in the non-lattice case $r_1=0$, we will not have more than one class;
in that case, we set $d=0$ and  $\Xx^{(1)} = \Xx$. 

\begin{thm}\label{thm:posrec}
Let $\mu$ be a probability measure on $\Z^{r_1}\times \R^{r_2}$ whose
lattice marginals $\mu_i$ ($i=1,\dots, r_1$) satisfy \eqref{eq:gcd}, while for
$i > r_1\,$, the marginals are non-lattice.  

Suppose that for each $i \in \{ 1, \dots, r\}$, the one-dimensional
reflected random walk induced by $\mu_i$ is positive recurrent on the respective
set $\Ll_i$ according to \eqref{eq:Ll}.

Then each class\/ $\Xx^{(j)}$ of \eqref{obs:cosets} carries a unique 
invariant probability measure $\nu^{(j)}$ for the 
$r$-dimensional reflected random walk induced by $\mu$. Reflected
random walk started in any point of\/ $\Xx^{(j)}$ is a.s. absorbed by\/ $\Ll^{(j)} = \supp(\nu^{(j)})$,
and it is positive recurrent on\/ $\Ll^{(j)}$.
\end{thm}

\begin{proof} If $r_1=0$ then the proof simplifies, as we shall clarify at the
end. So assume $r_1 \ge 1$. As in 
\eqref{eq:stop}, we consider the a.s. finite stopping times 
\begin{equation}\label{eq:stop2}
\begin{aligned}
\taub(0)=0\AND
\taub(n) &= \inf \{ k > \taub(n-1) : \pib(X_k^x) = \pib(x) \}\\ &=  \inf \{ k > \taub(n-1) : \pib(S_k) = \zero \}\,,
\end{aligned}
\end{equation}
where again $S_k = Y_1 + \dots + Y_k \in \Z^{r_1} \times \R^{r_2}$. Once more, the increments
$\taub(n) - \taub(n-1)$, $n \ge 1$, are i.i.d.
The stationary probability distribution of $\bigl(\pib(X_n^x)\bigr)$ on $\Gamma^{(j)}$ is uniform, 
whence $\Ex\bigl(\taub(1)\bigr)=|\Gamma|$. We look at the induced process
$(X_{\taub(n)}^x)_{n \ge 0}$ on each set $\Xx_{\epb}\,$, where $\epb \in \{0,1\}^{r_1}$. 
As in \eqref{eq:induced}, it is an SDS induced by the i.i.d. 
multidimensional contractions
\begin{equation}\label{eq:ind-mult}
\begin{aligned}
F_n &= f_{Y_{\tb(n)}} \circ f_{Y_{\tb(n)-1}} \circ \dots \circ f_{Y_{\tb(n-1)+1}}\,,\quad\text{with}\\
F_n(x_1, \dots, x_r) &= \Bigl( F_{n,1}(x_1), \dots, F_{n,r}(x_r)\Bigr)\,, \quad\text{where}\\
F_{n,i} &= f_{Y_{\tb(n)-1,i}} \circ \dots \circ f_{Y_{\tb(n-1)+1,i}}\,.
\end{aligned}
\end{equation}
Here, $Y_{k,i}$ is of course the $i$-th coordinate of the random vector $Y_k\,$, and as above 
$f_{b}(x_i) = |x_i-b|$ for $b, x_i \in \R$. Note that the random
mappings $F_n$ do not depend on the point $x$ or the class $\Xx_{\epb}$ where the process starts.
By \eqref{eq:multi-contr}, the SDS $(X_{\taub(n)}^x)$ is strongly contractive on each $\Xx_{\epb}\,$.
We write $\Ll_{\epb}$ for its attractor. Hence, each of its marginal processes is 
also strongly contractive;  for any starting point, it is absorbed by its attractor, which 
is the respective projection of $\Ll_{\epb}\,$. (Here, ``absorbed'' means in the lattice case that with
probability $1$ it belongs to the attractor from some time onwards, while in the non-lattice case,
the distance to the attractor tends to $0$.)

\bigskip

\emph{Claim.} Each marginal process $(X_{\taub(n),i}^{x_i})_{n \ge 0}$ is positive recurrent
on its attractor.

\bigskip

In spite of being ``obvious'', this needs justification.

\smallskip
 
We start by considering the first marginal of $(X_n^x)$, which is driven
by the lattice distribution $\mu_1\,$. We can apply the reasoning of Lemma 
\ref{lem:evodd} and the subsequent lines to $(X_{n,1}^{x_1})$. 
Define
$$
\pib': \Xx \to \N_0 \times \{0,1\}^{r_1-1}\,,\quad \pib'(x_1, \dots, x_r) 
= \bigl(x_1, \pi(x_2), \dots, \pi(x_{r_1})\bigr).
$$
The process $\bigl(\pib'(X_n^x)\bigr)_{n \ge 0}$ is ``reflected random walk on $\N_0$ with
internal degrees of freedom''. Its transition probabilities are
\begin{equation}\label{eq:transprob'}
p'\bigl((x_1,\epb'),(y_1, \overline\epb') \bigr) = \Prob\bigl[\, |x_1 - Y_{1,1}|=y_1\,,\;
\bigl(\pi(Y_{1,2}),\dots,\pi(Y_{1,r_1})\bigr) = \overline\epb' - \epb'\,\bigr]\,,
\end{equation}
where of course $\overline\epb' - \epb'$ is taken mod $2$. Observation \ref{obs:cosets}
applies to $\bigl(\pib'(X_n^x)\bigr)$ if one replaces $\Xx^{(j)}$ with
$$
\pib'(\Xx^{(j)}) = \bigl\{ (x_1, \epb') : \bigl( \pi(x_1), \epb' \bigr) \in \Gamma^{(j)} \bigr\}.
$$
Since the transition probabilities \eqref{eq:transprob'} are additive mod $2$ in the 
$\epb'$-coordinates, an invariant measure with finite total mass for 
$\bigl(\pib'(X_n^x)\bigr)$ is given by
$$
\nu_1'(x_1,\epb') = \nu_1(x_1)\,,
$$
where $\nu_1$ is the invariant probability distribution for the first marginal process
driven by $\mu_1\,$. We let $\nu_1^{(j)}$ be the probability measure obtained by
restricting $\nu_1'$ to $\pi'(\Xx^{(j)})$ and normalising it. We shall see that 
$\supp(\nu_1^{(j)})$ is the only essential class of $\bigl(\pib'(X_n^x)\bigr)$
within $\pi'(\Xx^{(j)})$. 

In any case, $\bigl(\pib'(X_n^x)\bigr)$ is positive recurrent in the irreducible (whence 
essential) class of each point $(x_1,\epb')$ with $x_1 \in \supp(\nu_1)$.
We have $\pib(x) = \epb = (\ep_1, \epb')$, where $\pib'(x) = (x_1\,,\epb')$ and $\ep_1=\pi(x_1)$.
The stopping times 
$\taub(n)$ are the successive instants when 
$\bigl(\pib'(X_n^x)\bigr)$ visits the subset 
$(2\cdot \N_0 + \ep_1) \times \{\epb'\}$.
Thus, if $x$ is such that $x_1 \in \supp(\nu_1)$,  
then the return time of $\bigl(\pib'(X_n^x)\bigr)$ to $(x_1,\epb')$ has finite expectation.
At that return time, also $\bigl(\pib'(X_{\taub(n)}^x)\bigr)$ is back at $(x_1,\epb')$,
whence also the return time of $\bigl(\pib'(X_{\taub(n)}^x)\bigr)$ has finite expectation.
But the first marginal of $\bigl(\pib'(X_{\taub(n)}^x)\bigr)$
is just the first marginal of $(X_{\taub(n)}^x)$, so that the return time of the first marginal
process also has finite expectation.

\smallskip

This argument shows that all the lattice marginal processes $(X_{\taub(n),i}^{x_i})$,
$i=1, \dots, r_1$, are positive recurrent on their respective attractors (as we know that
they are strongly contractive, whence the respective attractor -- depending on $\pib(x)$ -- 
is the unique essential class). 

\smallskip

Now suppose that there are also non-lattice marginals, i.e., $r > r_1\,$.
Then we consider the last marginal of $(X_n^x)$, which is driven
by the non-lattice distribution $\mu_r\,$. We know from 
propositions \ref{pro:contractive} and \ref{pro:invmeas} that this marginal
SDS is strongly contractive with invariant probability measure $\nu_r\,$.
Its attractor is $\supp(\nu_r)$.

For any $x \in \Xx_{\ep}\,$, the $r^{\text{th}}$ marginal process 
$(X_{\taub(n),r}^{x_r})$ is a strongly contractive sub-SDS of $(X_{n,r}^{x_r})$.
This time we define
$$
\pib'': \Xx \to \{0,1\}^{r_1} \times \R_+\,,\quad 
\pib''(x)  = \bigl(\pib(x), x_r\bigr).
$$
The transition probabilities of the process $\bigl(\pib''(X_n^x)\bigr)_{n \ge 0}$ are
\begin{equation}\label{eq:transprob''}
p''\bigl((\epb, x_1), \{\overline\epb\}\times B \bigr) = \Prob\bigl[\, |x_r - Y_{1,r}| \in B\,,\;
\bigl(\pi(Y_{1,1}),\dots,\pi(Y_{1,r_1})\bigr) = \overline\epb - \epb\,\bigr]\,,
\end{equation}
again taking $\overline\epb - \epb$ mod $2$, where $B \subset \R_+$ is a Borel set.
Again, Observation \ref{obs:cosets}
applies to $\bigl(\pib''(X_n^x)\bigr)$ if one replaces $\Xx^{(j)}$ with
$$
\pib''(\Xx^{(j)}) = \Gamma^{(j)} \times \R_+\,.
$$
Once more, since the transition probabilities \eqref{eq:transprob''} are additive mod $2$ in 
the $\epb$-coordinates, an invariant
measure with finite total mass for $\bigl(\pib''(X_n^x)\bigr)$ is given by
$$
\nu_r''(\{\epb\}\times B) = \nu_r(B)\,,
$$
where $\nu_r$ is the invariant probability distribution for the $r^{\text{th}}$ marginal process
driven by $\mu_r\,$. That marginal process is strongly contractive, and its attractor is
$\supp(\nu_r)$.

The projected random walk $\bigl(\pib(X_n^x)\bigr)$ is positive recurrent on each of its 
irreducible classes $\Gamma^{(j)}$. If $\epb \in  \Gamma^{(j)}$ and $x \in \Xx_{\epb}$
then $\pi''(X_{\epb}) = \{\epb\} \times \R_+$ is a recurrent set for $\bigl(\pib''(X_n^x)\bigr)$. 
It is a straighforward and well-known consequence 
that the restriction of $\nu_r''$ to $\{\epb\} \times \R_+$ is an invariant measure
for the induced process on that recurrent set; see e.g. the proof of \cite[Lemma 2.6]{PeWo2}
(which at first yields execcisivity of the restriction, while invariance follows
from the fact that the restricted measure has finite total mass). Now, that induced
process is nothing but $\bigl( \epb, X_{\taub(n),r}^x \bigr)$. Therefore
$\nu_r$ is the unique invariant probability measure of $\bigl(X_{\taub(n),r}^x \bigr)$.
Since the latter process is strongly contractive, $\supp(\nu_r)$ is its attractor,
and the process is positive recurrent on that set.

\smallskip

Again, this argument applies to all non-lattice marginals of our SDS, and
the claim is proved. 

\smallskip

We know (via Proposition \ref{pro:invmeas}) that for every starting point $x \in \Xx$, each marginal 
SDS $(X_{\taub(n),i}^{x_i})$  has a unique invariant probability measure $\nu_{i, \epb}\,$ on its attractor,
which depends on $\epb=\pib(x)$. By Proposition
\ref{pro:reverse}, there is a non-negative integer, resp. real random variable $Z_{i, \epb}$ 
such that for the reversed process, we have  
$$
\wh X_{\taub(n),i}^{x_i} = F_{1,i} \circ F_{2,i} \circ \dots \circ F_{n,i}(x_i) \to Z_{i, \epb}
\quad\text{almost surely}
$$
for each $x=(x_1, \dots, x_r)  \in \Xx_{\epb}\,$, with the $F_{k,i}$ given in \eqref{eq:ind-mult} . But then we get that
$$
\wh X_{\taub(n)}^x = F_1 \circ F_2 \circ \dots \circ F_n(x) \to Z_{\epb} = (Z_{1, \epb}\,,\dots, Z_{r, \epb})
\quad\text{almost surely} 
$$
for each $x \in \Xx_{\epb}\,$. Since the limit random variable $Z_{\epb}$ does not depend on the
starting point, its distribution $\nu_{\epb}$ is an invariant probability measure for
$(X_{\taub(n)}^x)$, and $\Ll_{\epb} = \supp(\nu_{\epb})$. 
We note that the marginals of $\nu_{\epb}$ are the above measures 
$\nu_{i, \epb}\,$. (Recall here that for $r_1 < i \le r$, we have 
$\nu_{i, \epb} = \nu_i\,$,
the invariant probability measure for the reflected random walk driven by the marginal $\mu_i\,$.)

\smallskip

Now suppose that the starting point $x$ lies in $\Xx^{(j)}$. The projected random walk
$\bigl(\pib(X_n^x)\bigr)$ is positive recurrent on $\Gamma^{(j)}$. Therefore $(X_n^x)$ visits each 
$\Xx_{\epb} \subset \Xx^{(j)}$ infinitely often with probability $1$. Since the $\taub(n)$
are the times of the successive return visits to each of those $\Xx_{\epb}\,$,
we see that the set of accumulation points of  $(X_n^x)$ coincides almost
surely with 
\begin{equation}\label{eq:Lj}
\Ll^{(j)} = \bigcup_{\epb \in \Gamma^{(j)}} \Ll_{\epb}\,.
\end{equation}
We choose $\epb \in \Gamma^{(j)}$ and use $\nu_{\epb}$ to construct a probability
measure on $\Xx^{(j)}$ by
$$
\begin{aligned}
\nu^{(j)}(B) 
&= \frac{1}{\Ex\bigl(\taub(1)\bigr)} \int_{\Ll_{\epb}} 
                   \Ex \Biggl( \sum_{n=0}^{\taub(1)-1} \ch_B(X_n^x)\Biggr) d\nu_{\epb}(x) \\
&= \frac{1}{|\Gamma|} \sum_{n=0}^{\infty}
         \int_{\Ll_{\epb}} \Prob[X_n^x \in B\,,\; \taub(1) \ge n+1]\,  d\nu_{\epb}(x)\,,         
\end{aligned}
$$
where $B \subset \Xx^{(j)}$ is a Borel set. It is well known and easy to verify that
this is an invariant probability measure for $(X_n^x)$. 

Suppose that $\nu$ is an arbitrary invariant probability measure for $(X_n^x)$ on $\Xx^{(j)}$.
Every point in $\Xx^{(j)} \setminus \Ll^{(j)}$,
not being an accumulation point of $(X_n^x)$, is transient (has a neighbourhood which is 
visited only finitely often). Thus, we must have $\supp(\nu) \subset \Ll^{(j)}$.
On the other hand, invariance of $\nu$ implies that $X_1^x \in  \supp(\nu)$ a.s.
for any $x \in  \supp(\nu^{(j)})$, and iterating, the entire trajectory of $(X_n^x)$ is
in $\supp(\nu)$. We see that
$\supp(\nu) = \Ll^{(j)}$.

The projected probability measure $\pib(\nu)$ must be invariant for the factor chain 
$\bigl(\pib(X_n^x)\bigr)$ in $\Gamma^{(j)}$. Therefore $\nu(\Xx_{\epb}) = 1/|\Gamma|$
for every $\epb \in \Gamma^{(j)}$. 
It is again a well-known fact that the normalised restriction of $\nu$ to $\Xx_{\epb}$
must be the (as we know, unique) invariant probability measure for the induced 
process $(X_{\taub(n)}^x)$ on that set. Thus, $\nu = \nu^{(j)}$ is unique, 
$$
\nu^{(j)} = \frac{1}{|\Gamma|} \sum_{\epb \in \Gamma^{(j)}}  \nu_{\epb}\,,
$$
where $\nu_{\epb}$ is viewed as a measure on the whole of $\Xx^{(j)}$.
This concludes the proof in the presence of lattice marginals.

\medskip

In the purely non-lattice case when $r_1=0$, we do not need to pass to an induced 
subsystem which becomes contractive: the reversal argument applies directly to
the original reflected random walk. Indeed, this is the case treated by \cite[Th. 24]{Pe}.
\end{proof}

We know that the one-dimensional marginals of each of the invariant probability measures 
$\nu^{(j)}$ on the different parts $\Xx^{(j)}$ of the state space
are the invariant measures $\nu_i$ of the marginal processes, which are supported by the
intervals $[0\,,N_i] \cap \R_+\,$, resp. $[0\,,N_i] \cap \N_0\,$.  
In the higher-dimensional case, the essential classes $\Xx^{(j)}$
where the reflected random walk takes place -- the respective support of the $\nu^{(j)}$ 
- are not easily determined. We illustrate this by the following simple examples.

\begin{exs} We let $\Xx = \N_0^2$.
\\[3pt]
(a) Let $\mu = \frac12(\delta_{(2,3)} + \delta_{(3,2)})$. Then $N_1 = N_2 = 3$
and the reflected random walk is absorbed by (a subset of) $\{0,1,2,3\}^2$.
We have $\Gamma = \{0,1\}^2$, and there is only one essential class.
Indeed, there are the three irreducible classes
$$
\{(0,0), (2,3), (3,2)\}\,,\quad \{(3,3)\} \AND  
\{0,1,2,3\}^2 \setminus \{(0,0), (2,3), (3,2), (3,3) \}.
$$
The latter is the essential one.
\\[3pt]
(b) Let $\mu = \frac12(\delta_{(-1,2)} + \delta_{(2,-1)})$. Then $N_1=N_2=\infty$
and $\Gamma = \{0,1\}^2$. Again, there is only one essential class,
and one finds that this is $\N_0^2 \setminus \{(0,0)\}$.
\\[3pt]
(c) Let $\mu = \frac12(\delta_{(-1,3)} + \delta_{(3,-1)})$. Again, $N_1=N_2=\infty$
but $\Gamma = \{ (0,0), (1,1)\}$. Reflected random walk evolves on
the two separated parts 
$$
\Xx^{(1)} = \{ (k,l) \in \N_0^2 : k+l \;\text{is odd}\,\}
\AND
\Xx^{(2)} = \{ (k,l) \in \N_0^2 : k+l \;\text{is even}\,\}.
$$
While the whole of $\Xx^{(1)}$ is an essential class and thus equal to $\Ll^{(1)}$,
the essential class within $\Xx^{(2)}$ is  $\Ll^{(2)} = \Xx^{(2)} \setminus \{(0,0)\}\,$. 
\\[3pt]
One can also find examples as in (b) or (c) where a bigger region around the origin
is not part of the attractor.\qed 
\end{exs}

\begin{rmks}\label{rmk:noteasy}
(a) In view of Proposition \ref{pro:rec}, the sets $\Ll_{\epb}$ only depend on $\supp(\mu)$,
and thus also the set $\Ll^{(j)}$ of \eqref{eq:Lj} does not depend on recurrence, but just
on $\supp(\mu)$. And as long as all marginals satisfy $\mu_i\bigl((0\,,\,\infty)\bigr) > 0$,
we can modify $\mu$ to obtain another probability measure with the same support that induces 
a reflected random walk which is positive recurrent on each $\Xx^{(j)}$ 
(or, more precisely, $\Ll^{(j)}$).
\\[6pt]
(b) There is a very simple argument, communicated to us by Nina Gantert, which shows
at least in the discrete case  $(r_2=0)$ that positive recurrence of each of the
marginal processes implies that RRW starting from any point in $\N_0^r$ must be absorbed
by a positive recurrent essential class. We display that argument here, for simplicity
taking only $r=2$. 
There must be finite sets $A_1\,, A_2 \subset \N_0$ such that 
$\nu_1(A_1) + \nu_2(A_2) > 1$, where the $\nu_i$ are the respective
stationary probability measures. Then for $x \in \N_0^2$, by the convergence theorem,
$$
\begin{aligned}
\frac{1}{n}\sum_{k=0}^{n-1}\Prob[ X_k^x \in A_1 \times A_2] 
&\ge \frac{1}{n}\sum_{k=0}^{n-1}\Bigl(
\Prob[X_{k,1}^{x_1} \in A_1] + \Prob[X_{k,2}^{x_2} \in A_2]-1\Bigr)\\
&\to \nu_1(A_1) + \nu_2(A_2) - 1 > 0\,, \quad\text{as }\; n \to \infty\,.
\end{aligned}
$$
Thus, one would think that the first issue
is to use purely algebraic arguments involving only $\supp(\mu)$ which should lead to
a description of the essential classes of RRW, showing that there is precisely one within
each $\Xx^{(j)}$. However, to the authors it is by no means obvious how to achieve this without
involving the local contractivity arguments used above. Indeed, already in the one-dimensional
case, without use of local contractivity (which works via the algebraic Proposition \ref{pro:even}),
the corresponding reasoning is amazingly hard: quoting \cite[p. 100]{Bo2}, ``d'une surprenante 
difficult\'e$\,$'' -- even though in dimension 1 the stationary distribution is known
explicitly.\qed
\end{rmks}

The next question is whether one can get a more general recurrence result regarding null recurrence,
that is, when some of the marginal distributions give rise to null recurrent
reflected random walks; compare with propositions \ref{pro:refl-recurr} and \ref{pro:sqrt}.
This appears to be a hard task. We next show that in general, for recurrence one 
cannot have more
than two marginals which are only null recurrent.

Consider $\mu$  on $\R^r$. We take a sequence $(\ee_{n,i})_{n \ge 0, 1 \le i \le r}$ of i.i.d. 
random variables which are equidistributed on $\{ \pm 1\}$
and independent of $(Y_n)_{n \ge 1}\,$.
For each one-dimensional marginal $\mu_i$ and the associated coordinates 
$Y_{n,i}$ we consider the associated process 
$$
W_{0,i}^{x_i} =  x_i\,,  \AND W_{n+1,i}^{x_i} = W_{n,i}^{x_i} + E_{n,i}^{x_i} Y_{n+1,i}\,,
\quad\text{where}\quad
E_{n,i}^{x_i} = \begin{cases} -1\,,&\text{if}\; W_{n,i}^{x_i} > 0\,,\\
                      \ee_{n,i}\,,&\text{if}\; W_{n,i}^{x_i} = 0\,,\\
                      \; 1\,,&\text{if}\; W_{n,i}^{x_i} < 0\,.
        \end{cases}
$$
Then we have
$$
|W_n^x| = X_n^{|x|}\,,
$$
where (recall) absolute values are taken coordinate-wise. The following is a straightforward exercise.

\begin{lem}\label{lem:symm} If $\mu$ is \emph{fully symmetric,} that is,
invariant under all coordinate reflections $x_i \mapsto -x_i$ ($i = 1,\dots, d$),
then the $r$-dimensional increments
$$
\wt Y_n = E_{n-1}^{x} \cdot Y_n = \bigl(E_{n-1,1}^{x_1}Y_{n,1}\,,\dots, E_{n-1,d}^{x_d}Y_{n,d}\bigr)
$$
are i.i.d. $\mu$-distributed. In particular, for any $x \in \Xx$ and Borel set $B \in \R_+^r$,
\begin{equation}\label{eq:return-symm}
\begin{gathered}
\Prob[X_n^x \in B] = \Prob[x+S_n \in B^*]\,,\quad\text{where}\\ 
B^* = \{ (\pm y_1\,,\dots, \pm y_r) : (y_1\,,\dots, y_r) \in B \}.
\end{gathered}
\end{equation}
\end{lem}

We observe that when $\supp(\mu)$ is a fully symmetric set, then the induced reflected random walk
is such that $\Ll^{(j)} = \Xx^{(j)}$ for the essential classes given by \eqref{eq:Lj}, resp. the
respective partition \eqref{obs:cosets} of the state space $\Xx$.

\begin{cor}\label{cor:dim}
Suppose that $\mu$ is fully symmetric.
Then reflected random walk induced by $\mu$ is transient whenever the dimension
is $r \ge 3$.
When $r \in \{1,2\}$, a sufficient condition for recurrence is that
$\mu$ has finite moment of order $r\,$.
\end{cor}

We shall deduce from Theorem \ref{thm:r-s-rec}  below that this has the following generalisation.

\begin{cor}\label{cor:dim2} Let $\mu$ be a probability measure on $\R^{r+s}$
whose lattice marginals satisfy \eqref{eq:gcd}. Write $\mu_{\lfloor r \rfloor}$ for the 
$r$-dimensional
marginal of $\mu$ in the first $r$ coordinates and  $\mu_{\lceil s \rceil}$ for the $s$-dimensional
marginal of $\mu$ in the last $s$ coordinates, where $s \in \{ 1,2\}$. 
Suppose that the reflected random walk 
induced by $\mu_{\lfloor r \rfloor}$ is positive recurrent on each of its essential classes. 

If $\mu_{\lceil s \rceil}$ is fully symmetric and has finite moment of order $s$,
then the reflected random walk induced by $\mu$ is (topologically) null recurrent
on its essential classes.
\end{cor}

The following, regarding the joint observation of independent parts, is obvious.

\begin{lem}\label{lem:independent}
Suppose that  the probability measure $\mu$ on $\R^{r_1+r_2}$ is such that all lattice
marginals satisfy \eqref{eq:gcd} and
$$
\mu =  \mu_{\lfloor r_1 \rfloor}\otimes \mu_{\lceil r_2 \rceil}\,.
$$
If RRW driven by $\mu_{\lfloor r_1 \rfloor}$ is positive recurrent  and RRW driven 
by $\mu_{\lceil r_2 \rceil}$ is null recurrent,
then RRW  driven by $\mu$ is null recurrent (on the respective essential classes).

This holds in particular, when $r_2=1$ and one of the conditions of for null recurrence
of \S 3 is satisfied.
\end{lem}

The following provides a class of examples regarding null recurrence in dimension $2$.

\begin{lem}\label{lem:E-P}
Let $\mu_1$ and $\mu_2$ be probability measures on $\Z$ which satisfy \eqref{eq:gcd}.
Suppose they have exponential moments of all orders and are centred.
Then RRW on $\N_0^2$ induced by $\mu_1 \otimes \mu_2$ is null recurrent on its 
essential classes. 
\end{lem}

\begin{proof}
Under the above assumptions, it was shown by {\sc Essifi and Peign\'e}~\cite{EsPe}
that for all $x, y \in \N_0$ 
$$
\Prob[X_{n,i}^x = y] \sim C^{(i)}_y \, n^{-1/2}\quad \text{as }\; n \to \infty,
$$
where $C^{(i)}_y > 0$, for $i=1,2$. The statement follows.
\end{proof}

With weaker moment conditions, one can well have two independent RRWs, each of which
is null recurrent, while the resulting two-dimensional RRW is transient.

\begin{exa}\label{ex:sub}
On $\Z$, let $(Y_n)$ be equidistributed on $\{\pm 1 \}$, so that $S_n=Y_1 + \dots + Y_n$
is simple random walk. Let $\bigl(\tau(n)\bigr)_{n \ge 0}$ be a 
sequence of random times which is independent of $(Y_n)$ and such that
$\tau(0)=0$ and  $\tau(n) - \tau(n-1)$ are i.i.d. $\N$-valued. The associated
\emph{subordinated random walk} is 
$$
S_{\tau(n)} = \wt Y_1 + \dots + \wt Y_n\,,\quad \text{where}\quad
\wt Y_k = Y_{\tau(k-1) + 1} + \dots +  Y_{\tau(k)}\,.
$$
Now let $0 < \alpha < 1$ and consider $\tau(n) = \tau_{\alpha}(n)$, where 
$$
\Prob[\tau_{\alpha}(n)-\tau_{\alpha}(n-1) = k] = 
\frac{\alpha\,\Gamma(k-\alpha)}{k! \, \Gamma(1-\alpha)} \sim 
\frac{\alpha}{\Gamma(1-\alpha)}\, \frac{1}{k^{1+\alpha}} 
\quad (1 \le k \to \infty).
$$
By  {\sc Bendikov and Saloff-Coste}~\cite[Thm.3.4]{BeSa},
$$
\Prob[S_{\tau(2n)} = 0] \simeq n^{-\frac{1}{2\alpha}}\,,
$$
where $\simeq$ means asymptotic equivalence of sequences.
Let $\mu_{\alpha}$ be the distribution of $\wt Y_1\,$.
We see that $(S_{\tau(n)})$, the symmetric random walk on
$\Z$ with law $\mu_{\alpha}\,$, as well as the associated 
RRW on $\N_0$ are recurrent if and only if $\alpha \ge 1/2$.

Now consider $\mu = \mu_{\alpha} \otimes \mu_{\alpha}$ on
$\Z^2$. It is fully symmetric, and we get that for any
$\alpha \in (0\,,\,1)$, the random walk induced by $\mu$
with reflection in none, one or both coordinates is
transient.\qed
\end{exa}


\section{Reflected plus non-reflected coordinates}

We now consider the situation of \eqref{eq:process} in dimension $r+s$ with $s \in \{1,2\}$, and
state space as in \eqref{eq:statespace}. As before, we write $\mu_{\lfloor r\rfloor}$
and $\mu_{\lceil s \rceil}$ for the overall marginal distributions of $\mu$ in the first
$r$ and last $s$ variables, respectively.
 
\begin{thm}\label{thm:r-s-rec}
 Suppose that $\mu_{\lfloor r\rfloor}$ satisfies the assumptions of Theorem \ref{thm:posrec},
so that the associated reflected random walk $(X_n^x)$ on 
$\N_0^{r_1}\times \R_+^{r_2}$ is positive recurrent. 

If  $\mu_{\lceil s \rceil}$ has finite moment of order $s$, then the process $(X_n^x, Z_n)$ 
of \eqref{eq:process}
is (topologically) recurrent if and only if $\mu_{\lceil s \rceil}$ is centred.
\end{thm}

Here, we mean that when $\Ll^{(j)}$ is one of the essential classes \eqref{eq:Lj}
of  $(X_n^x)$ according to Theorem \ref{thm:posrec}, then for each $x \in \Ll^{(j)}$
and $v \in \R^s$, the process $(X_n^x, v + Z_n)$ returns to any neighbourhood of
$(x,v)$ infinitely often with probability $1$. Of course, when  $\mu_{\lceil s \rceil}$ is lattice,
there are infinitely many returns to $(x,v)$ itself. Note that we may assume w.l.o.g.
that $v=0$. We also remark here that $(X_n^x\,,Z_n)$ is a typical case of a 
\emph{Markov random walk} or \emph{random walk with internal degrees of freedom} with positive
recurrent \emph{driving Markov chain} $(X_n^x)$. There is an ample literature on processes
of this type, see e.g. {\sc Jacod}~\cite{Jac}, {\sc Kr\'amli and Sz\'asz}~\cite{KrSz}, 
{\sc Babillot}~\cite{Bab} or {\sc Uchiyama}~\cite{Uch} and the references in those
papers. 

\begin{proof}
Because this is considerably simpler, we first consider the case when  $\mu_{\lfloor r\rfloor}$ 
is purely lattice, that is, $r_2=0$.
Let $x \in \Ll^{(j)}$, and let $\tb(n)$ be the successive return times
of $(X_n^x)$ to $x$, with $\tb(0)=0$. They have i.i.d. increments  with
finite expectation by positive recurrence. Then 
$$
Z_{\tb(n)} = \wt V_1 + \dots + \wt V_n\,,\quad\text{where}   \quad \wt V_k = Z_{\tb(k)}-Z_{\tb(k-1)}\,,
$$
and the $\wt V_k$ are i.i.d. By Wald's identity, 
$$
\Ex(\,\wt V_1) = \Ex\bigl(\tb(1)\bigr) \,\Ex(V_1)\,,
$$
and if $V_1$ has finite second moment then so does $\wt V_1\,$.
The result follows.

\smallskip

The situation is more complicated when the reflected part is not
purely lattice. In this case, we start with a compact neighbourhood $U$ of
some point $x \in \Ll^{(j)}$. 
We know that for any $y \in U$, the chain $(X_n^y)$ returns
to $U$ almost surely. Thus, we can consider the induced process $(X_{\tb(n)}^y)$ on $U$,
where $\tb(n)$ are the times of the successive visits to $U$. Note that they depend on the
starting point $y$ and do not have i.i.d. increments. In any case, it is a well known fact
that the normalised restriction $\nu_U=\nu^{(j)}_U$ of the invariant probability measure 
$\nu^{(j)}$ to $U$ is an invariant probability for the induced process. 

We shall use a method of \cite[Thm. 4.1]{BaBoEl}.
For any probability distribution $\nu$ supported in $\Ll^{(j)}$, we let 
$\Prob_{\nu} = \Prob_{(\nu,0)}$ be the probability on the trajectory space of 
$(X_n^x, v+Z_n)$,  where $(X_n^x)$  has starting distribution $\nu$ -- so that 
we might as well use the notation $(X_n^{\nu})$ -- while $(S_n)$ starts at $v= 0$. In other words,
$$
\Prob_{\nu} = \int \Prob_{(x,0)} \, d\nu(x)\,,
$$
where in general $\Prob_{(x,v)}$ refers to starting the process at the deterministic point $(x,v)$.

Since $\nu^{(j)}$ is the unique invariant probability for the original process on 
$\Xx^{(j)}$, resp. $\Ll^{(j)}$, also $\nu_U$ is the unique invariant probability 
for the induced process.  Therefore  the induced process on $\Ll^{(j)}$ with initial 
distribution $\nu_U$ -- which we denote by $(X_{\tb(n)}^{\nu_U})$ -- is not only stationary, 
but ergodic under $\Prob_{\nu_U}\,$ -- see e.g. {\sc Hern\'andez-Lerma and 
Lasserre}~\cite[Prop. 2.4.3]{HeLa}

\smallskip

\noindent
\emph{Claim.} The sequences of random variables $(\wt V_n)$ and $\bigl(\tb(n)-\tb(n-1)\bigr)$ are
stationary ergodic under $\Prob_{\nu_U}\,$.

\smallskip

\noindent
\emph{Proof of the claim.} Stationarity is straightforward, and contained in the
first part of the following. 
Let $(\mathcal{F}_n)$ be the filtration
of the $\sigma$-algebra on the trajectory space generated by $(X_n^x, Z_n)$. 
Take a measurable function $\phi: \R^{\N} \to  \R^+$. Then, since the distribution
of  $X_{\tb(n)}^{\nu_U}$ is $\nu_U$ and the transitions of $(Z_n)$ are translation invariant,
%
%
%
$$
\begin{aligned}
\Ex_{\nu_U}\bigl( \phi(\wt V_{n+1}\,, \wt V_{n+2}\,,\dots)\mid \mathcal{F}_{\tb(n)}\bigl) 
&= \Ex_{(X_{\tb(n)}^{\nu_U}, Z_{\tb(n)})}\bigl(\phi(\wt V_1\,, \wt V_2\,, \dots)\bigr)\\
&= \Ex_{(X_{\tb(n)}^{\nu_U}, 0)}\bigl(\phi(\wt V_1\,, \wt V_2\,, \dots)\bigr)
= \Ex_{\nu_U}\bigl( \phi(\wt V_1\,, \wt V_2\,,\dots)\bigl)\,.
\end{aligned}
$$
Now suppose in addition that $W =  \phi(\wt V_1\,, \wt V_2\,,\dots)$ is measurable
with respect to the invariant $\sigma$-algebra of $(\wt V_1\,, \wt V_2\,,\dots)$,
so that $W = \phi(\wt V_{n}\,, \wt V_{n+1}\,,\dots)$ for each $n$. Then by martingale
convergence and the above,
$$
W = \lim_{n\to \infty}\Ex_{\nu_U}(W \mid \mathcal{F}_{\tb(n)}) 
=  \lim_{n\to \infty}\Ex_{(X_{\tb(n)}^{\nu_U}, 0)}(W)\,.
$$
Therefore $W$ is also an invariant function of $(X_{\tb(n)}^{\nu_U})$, which is ergodic,
so that $W$ is $\Prob_{\nu_U}$-almost surely constant. This shows ergodicity
of $(\wt V_n)$. The proof for the increments $\bigl(\tb(n)-\tb(n-1)\bigr)$ is analogous.


\smallskip

Having proved the Claim, we recall that
as in the lattice case $\Ex\bigl(\tb(1)\bigr) < \infty\,$,
and by Wald's identity $\Ex(\,\wt V_1) =  0$ if and only $\Ex(V_1)=0$.

\smallskip

If $s=1$, then we see that under $\Prob_{\nu_U}$, the random walk $(Z_{\tb(n)})$ on $\R$ arises
from the sums of the stationary ergodic sequence of the random variables $\wt V_n\,$, which
have finite expectation.
By a theorem of \cite{At},   $(Z_{\tb(n)})$ is recurrent (= returns
infinitely often to any neighbourhood of $0$ with probability $1$) if and only if $V_1$ is centred.
This proves that $(X_n^x, v+Z_n)$ is recurrent (where RRW is considered one one of
its attractors $\Ll^{(j)})$ if and only if $\Ex(V_1)=0$.
 
If $s=2$, then our assumption is that $\Ex(V_1^2) < \infty$, so that $Z_n$ satisfies the 
Central Limit Theorem. If $\Ex(V_1) \ne 0$ then we have of course transience.
So suppose that $\Ex(V_1)=0$. Then $Z_n/\sqrt{n}$ converges in law to a non-degenerate $2$-dimensional
centred normal distribution. By Birkhoff's Ergodic Theorem, $\tb(n)/n \to \Ex\bigl(\tb(1)\bigr)$
almost surely under $\Prob_{\nu_U}\,$. Then, by an old theorem of {\sc R\'enyi}~\cite{Re} 
(going back to {\sc Anscombe}~\cite{An}), also
$Z_{\tb(n)}/\sqrt{\tb(n)}$ is asymptotically normal with the same limit distribution. 
Now we can apply the theorem of \cite{Schm}
to deduce that $(Z_{\tb(n)})$ is recurrent. This concludes the proof.
\end{proof}

\begin{proof}[Proof of Corollary \ref{cor:dim2}] Let $(X_n^x)$ be RRW induced by $\mu_{\lfloor r \rfloor}$ and $(v+Z_n)$
be the ordinary random walk induced by $\mu_{\lceil s \rceil}$. 
By Theorem \ref{thm:r-s-rec}, the process $(X_n^x, v+S_n)$ is recurrent on its
essential classes. A straighforward adaptation of Lemma \ref{lem:symm} yields
that we also have recurrence when there is reflection in the last two coordinates.
\end{proof}

Note that the last phrase of the proof remains true also when $s=2$ and there
only is reflection in one of the last two coordinates, while the other coordinate remains
non-reflected.
This observation together with Corollary \ref{cor:dim2} and Theorem \ref{thm:r-s-rec}
clarifies that there can not be a general result on recurrence with more than
two null-recurrent coordinates, be they reflecting or ``free''.

We conclude with an open problem. Suppose that $r=s=1$, so that we have reflection in
the first coordinate only, and no reflection in the second one. Also suppose that the second
marginal gives rise to a recurrent (ordinary) random walk (e.g.,
having finite first moment and being centred.) Provide general recurrence criteria, when
the reflected process in the first coordinate is null-recurrent. 

\end{document}